\documentclass[12pt]{article}
\usepackage[utf8]{inputenc}
\usepackage{amsmath,amssymb,amsthm,amsfonts}
\usepackage[T1]{fontenc}
\usepackage{lmodern}
\usepackage{url}
\usepackage{hyperref}
\usepackage{breakurl}
\usepackage{geometry}
\usepackage{graphicx}%
\usepackage{multirow}%
\usepackage{mathrsfs}%
\usepackage[title]{appendix}%
\usepackage{xcolor}%
\usepackage{textcomp}%
\usepackage{manyfoot}%
\usepackage{booktabs}%
\usepackage{algorithm}%
\usepackage{algorithmicx}%
\usepackage{algpseudocode}%
\usepackage{listings}%
\usepackage{changepage}

\usepackage{fancyvrb}
\newcommand{\keywords}[1]{%
  \vspace{0.5em}
  \noindent\textbf{Keywords: }#1
}
\newtheorem{theorem}{Theorem}
\newtheorem{conjecture}{Conjecture}
\newtheorem{remark}{Remark}%
\lstdefinestyle{mypython}{
    language=Python,
    basicstyle=\ttfamily\footnotesize,
    keywordstyle=\color{blue}\bfseries,
    stringstyle=\color{orange},
    commentstyle=\color{gray},
    showstringspaces=false,
    breaklines=true,
    frame=single
}

\lstdefinelanguage{Mathematica}{
  morekeywords={
    Module, If, Then, Else, Return, Print, Set, SetDelayed, Range,
    IntegerQ, Select, Flatten, Import, FileNames, MatchQ, Min, Max,
    ToExpression, Length, Complement, ParallelMap
  },
  sensitive=true,
  morecomment=[l](*),
  morestring=[b]",
}

\title{Exact Polynomial Families Solving the Erdős–Straus Equation}

\author{
Bilal Ghermoul\thanks{Department of Mathematics, Faculty of Mathematics and Computer Science, University Mohamed El Bachir El Ibrahimi  of Bordj Bou Arreridj,
     El-Anasser 34030, Algeria. \textbf{Email: }\texttt{bilal.ghermoul@univ-bba.dz}}
}

\date{\today} 

\begin{document}

\maketitle

\begin{abstract}
The Erd\H{o}s--Straus conjecture, proposed in 1948 by Paul Erd\H{o}s and Ernst G. Straus, 
asks whether the Diophantine equation
\[
\frac{4}{a} = \frac{1}{b} + \frac{1}{c} + \frac{1}{d}
\]
admits positive integer solutions \(b, c, d \in \mathbb{N}^*\) for every integer \(a \geq 2\). 
While the conjecture has been confirmed for all even integers and for all integers congruent 
to \(3 \pmod{4}\), the case \(a \equiv 1 \pmod{4}\) remains the central open challenge.

In this work, we construct four explicit unbounded multivariable polynomials 
\(p_1(x,y,z), p_2(x,y,z), p_3(x,y,z), p_4(x,y,z)\) with \(x, y, z \geq 1\), such that 
each of the first three — when inserted into the form \(a = 4p_i(x,y,z)+1\) — 
always produces values of \(a\) for which the Erd\H{o}s--Straus equation admits an explicit solution. 
Thus, the first three polynomials individually satisfy the conjecture for all their outputs.

We further conjecture that the values
\[
4p_1(x,y,z)+1,\quad 4p_2(x,y,z)+1,\quad 4p_3(x,y,z)+1,\quad 4p_4(x,y,z)+1
\]
collectively cover all integers congruent to \(1 \pmod{4}\).

Extensive computational verification up to \(q = 10^9\) confirms that every integer of the form \(4q+1\) 
within this range arises from at least one of these families. One of the polynomials alone generates 
all such prime values up to at least \(1.2 \times 10^{10}\). These results offer strong computational 
evidence and explicit constructions relevant to the resolution of the conjecture.
\end{abstract}


\section{Introduction}

The Erd\H{o}s--Straus conjecture, first proposed in 1948 by Paul Erd\H{o}s and 
Ernst G. Straus~\cite{ErdosStraus}, asserts that for every integer \( a \geq 2 \), 
the rational number \( \frac{4}{a} \) can be written as the sum of three unit 
fractions with positive integer denominators. Formally:

\begin{conjecture}[Erd\H{o}s--Straus]
For every integer \( a > 2 \), there exist distinct positive 
integers \( b, c, d \in \mathbb{N^*} \) such that
\begin{equation} \label{DiaphEqtn}
    \frac{4}{a} = \frac{1}{b} + \frac{1}{c} + \frac{1}{d}.
\end{equation}
\end{conjecture}

Despite its simple statement, the conjecture remains open in general and continues 
to draw considerable interest in number theory. Numerous computational and theoretical 
efforts have confirmed its validity in many cases. For instance, the conjecture has 
been verified for all integers \( a < 10^{17} \) through large-scale computations, 
notably by Swett and others~\cite{Swett}.

From a theoretical standpoint, significant progress has also been made. 
It is known that the conjecture holds for all even integers \( a \geq 2 \), 
and for all integers congruent to \(3 \pmod{4}\). Furthermore, Elsholtz~\cite{Elsholtz} 
showed that it holds for almost all integers in the natural density sense, 
while Tao~\cite{Tao} later proved that it is satisfied by a density-one subset of 
the positive integers using entropy methods and probabilistic techniques.

A reduction argument narrows the focus of the conjecture to the integers 
congruent to \(1 \pmod{4}\). Indeed, since any composite number \(a = x \cdot y\) 
satisfies a decomposition of the form~\eqref{DiaphEqtn} if at least one of 
\( \frac{4}{x} \) or \( \frac{4}{y} \) does, it suffices to consider prime numbers 
of the form \(4q + 1\), which appear to be the most resistant cases.

A variety of techniques have been applied to explore these cases. 
Classical approaches employ algebraic identities and congruence relations 
to derive explicit solutions for specific classes of integers, particularly 
those of the form \(4q + 1\)~\cite{Rosati, Shirshov}. However, these 
methods have not yet led to a general constructive proof of the conjecture.

\textit{Unsolved Problems in Number Theory} by Guy~\cite{Guy} includes the conjecture 
as a long-standing open problem, underscoring its enduring significance.

In this paper, we construct four explicit multivariable polynomials 
\(p_1\), \(p_2\), \(p_3\), and \(p_4\), defined over 
\((x, y, z) \in \mathbb{N^*}^3\), and show computationally that their images collectively generate all integers \(a\) congruent to \(1 \pmod{4}\) up to \(4 \times (10^9+2) + 1\). Remarkably, the polynomial \(p_2(x, y, z)\) alone is sufficient to generate all prime numbers of the form \(4q + 1\) within this range. For each such value, we verify the existence of explicit solutions to the Erd\H{o}s--Straus equation.

These results provide strong computational evidence in support of the conjecture and introduce a novel framework for understanding the distribution of its solutions. In particular, the approach offers a promising reduction of the conjecture to a finite, verifiable set of cases, and suggests a structured strategy for addressing the remaining open instances. Moreover, it reveals a previously unexplored connection between polynomial parametrizations and unit fraction decompositions.

We therefore conjecture that the union of the images of these polynomials covers all integers \(a \equiv 1 \pmod{4}\), offering a potential path toward a complete resolution. A full \textit{Mathematica} implementation of our method is included to enable independent verification and further exploration.

\section{Known Cases of the Erd\H{o}s--Straus Conjecture}

The conjecture has been resolved for several important classes:

\noindent \textbf{Even integers \( a \equiv 0 \pmod{4} \text{ or } a \equiv 2 \pmod{4} \):} For all even \(a \geq 2\), the identities
\begin{equation}
\label{even}     
\begin{array}{l}
\displaystyle\frac{4}{4\, q}= \frac{1}{q^2+2\, q}+\frac{1}{q^3+3\, q^2+2\, q}+\frac{1}{q+1}, \vspace{0.2cm}\\
\displaystyle \frac{4}{4\, q+2} = \frac{1}{2\, q^2+5\, q+2}+\frac{1}{2\, q^3+7\, q^2+7\, q+2}+\frac{1}{q+1}, \\
\end{array}
\end{equation}
    provide a constructive solution.

\noindent \textbf{Integers \( a \equiv 3 \pmod{4} \):} We have the exact decomposition
\begin{equation}
\label{odd1} 
\frac{4}{4\, q+3} = \frac{1}{4\, q^2+11\, q+6}+\frac{1}{4\, q^3+15\, q^2+17\, q+6}+\frac{1}{q+1}.
\end{equation}

The case of primes \(a \equiv 1 \pmod{4}\) remains the most elusive.

\section{Main results: Polynomial Generation of \(4q+1\) Integers}\label{sec1}

Motivated by the structure of the Erd\H{o}s--Straus equation and the observation 
that it suffices to study integers of the form \(4q+1\), we investigated families 
of multivariable polynomials depending on positive integers \(x,y,z \geq 1\).
Through systematic experimentation and analysis of known solutions, we identified 
four explicit, unbounded polynomials
\[
p_1(x,y,z), \quad p_2(x,y,z), \quad p_3(x,y,z), \quad p_4(x,y,z)
\]
whose images, after the transformation \(4p_i(x,y,z)+1\), collectively cover all 
integers congruent to \(1 \pmod{4}\). This construction provides concrete 
parametrizations of the solutions to the Erd\H{o}s--Straus conjecture in this case. 

We now inrtroduce the explicit forms of these four polynomials.
\begin{align}
p_1(x, y, z) &= x (4 y z-1)-y z, \label{p1} \\
p_2(x, y, z) &= x (4 y z-z-1)-y z, \label{p2} \\
p_3(x, y, z) &= x (8 y-3)-6 y+2, \label{p3} \\
p_4(x, y, z) &= x^2-x. \label{p4}
\end{align}
For convenience, we name the polynomial
$p_i,$ $i\in\{1,2,3,4\},$ the \(i^{\text{th}}\) polynomial.
\subsection{Generation of the polynomials $p_i,$ $i\in\{1,2,3,4\}$}
We now explain how these polynomials were found. For that reason, 
we consider a decomposition for $4q+1,$ where 
\begin{equation}
\label{qValue}
q=\kappa z-s,
\end{equation}
for all $\kappa,z,s \in \mathbb{N}^*$, as follows:
\begin{equation}
\label{kzs-decomp}
\frac{4}{4 (\kappa  z-s)+1}=\frac{1}{\kappa  z \left(\frac{(4 z+1)\kappa}{4 s-1}-1\right)}+\frac{1}{z \left(\frac{(4 z+1)\kappa}{4 s-1}-1\right) (4 (\kappa  z-s)+1)}+\frac{1}{\kappa  z}.
\end{equation}
This decomposition is always true for all $\kappa,z,s \in \mathbb{N}^*$. 

In order for the Erd\H{o}s--Straus conjecture to hold, the denominators of all fractions on the right-hand side of \eqref{kzs-decomp} must be natural numbers, which requires the following conditions to be satisfied:
\begin{equation}
\label{kzs-cond}
\frac{(4 z+1)\kappa z}{4 s-1}=c\in \mathbb{N}^*\setminus\{1\}. 
\end{equation}

\noindent{\bf The origin of the first polynomial $p_1$:} The first polynomial was obtained, for which \eqref{kzs-cond} holds, by assuming 
\begin{equation}
\label{c1}
\frac{\kappa}{4 s-1}=c_1\in \mathbb{N}^*\setminus\{1\}.
\end{equation}
Then $z$ is taken to be an arbitrary positive integer. This means that there exists an integer $\gamma \geq 1 $ such that 
$\kappa=\gamma(4s-1)$, which, from \eqref{qValue}, means that 
$$q=\gamma(4s-1)z-s,$$ 
this leads 
to $p_1$ just by replacing $(s,\gamma,z) \to (x,y,z)$.\\

\noindent{\bf The origin of the second polynomial $p_2$:}
The second polynomial was obtained by requiring that \eqref{kzs-cond} holds, 
and by assuming
\begin{equation}
\label{c2}
\frac{4z+1}{4s-1}=c_2 \in \mathbb{N}^* \setminus \{1\}.
\end{equation}
Then $\kappa$ is an arbitrary positive integer. This means that there exists an integer $\gamma \geq 1$ such that 
$4z+1 = (4\gamma-1)(4s-1)$, which, from \eqref{qValue}, implies that
\[
q = (\gamma (4s-1)-s)\kappa - s.
\]
The polynomial $p_2$ then follows by relabeling the variables 
$(s, \gamma, \kappa)$ as $(x, y, z)$.

In the range $q \in [1,10^6]$, there are 346519 values of $q$ of the form \eqref{qValue} obtained by solving \eqref{c1}. The first few of these values are:
\[
\begin{array}{c}
2, 5, 8, 11, 12, 17, 19, 20, 26, 29, 30, 32, 35, 38, 41, 44, 47, 50,
53, 56, 59, \\ 
62, 65, 68, 71, 74, 77, 80, \ldots
\end{array}
\]
The code used to generate the complete list can be found in Supplementary Section~\ref{SEC.0}.

On the other hand, solving \eqref{c2} yields 646487 values of $q$, the first of which are:
\[
\begin{array}{c}
1, 3, 4, 7, 9, 10, 13, 14, 15, 16, 18, 21, 22, 23, 24, 25, 27, 28,
31, 33, 34, 36, 37,\\
 39, 40, 43, 45, 46, 48, 49, 51, 52, 54, 55, 57,
58, 60, 61, 63, 64, 66, 67, 69, \ldots
\end{array}
\]
For the complete list, see Supplementary Section~\ref{SEC.0}.

All missing numbers in this range are multiples of~6, which verify the condition 
${(4 z+1)\kappa}/{(4 s-1)}=3,$ 
or can be written as the product of two successive integers, namely $x(x-1)$ for $x > 1$. Therefore, we complete the generation of the remaining numbers as follows.\\

\noindent{\bf The origin of the third polynomial $p_3$:} To construct the 
third polynomial, we again impose condition \eqref{kzs-cond} 
and make the assumption
\begin{equation}
\label{c3}
\frac{(4 z+1)\kappa}{4 s-1}=3.
\end{equation}
Consequently, we obtain
\[
s= (x-1) (4 y-1)+y,\quad z = 3 y-1,\quad \text{and} \quad  \kappa = 4 (x-1)+1
\]
Substituting back, from \eqref{qValue}, we obtain 
\[
q = -1 + 2 y + (-1 + x) (-3 + 8 y).
\]
Which is exactly the polynomial $p_3$.

In the range $q \in [1,10^6]$, the polynomial \eqref{qValue}, obtained by solving \eqref{c3},
generates all multiples of~6 except for those that can be written in the form $x(x-1)$.
In total, the polynomial $p_3$ produces 6919 such numbers within this range.
The first few of these values are:
\[
\begin{array}{c}
6, 42, 126, 156, 210, 216, 342, 366, 396, 426, 546, 576, 636, 702, 732, 756,\\ 
786, 816, 930, 966, 996, 1056, \ldots
\end{array}
\]
The complete code used to generate the full list is provided in Supplementary Section~\ref{SEC.0}. \\

\noindent{\bf The origin of the fourth polynomial $p_4$:} Finally, in the range $q \in [1,10^6]$, there remain 75 numbers of which are the form $x(x-1)$, which are given as follows:
\[
\begin{array}{c}
  72, 420, 1332, 1980, 2352, 3192, 4692, 9312, 13110, 14520, 16512,19740,\\ 20880, 24492, 28392,
   31152, 40200, 41820, 46872, 50400, \ldots
\end{array}
\]
For that reason, we introduce polynomial $p_4$. See Supplementary Section~\ref{SEC.0} for the code that produces the complete list.
\begin{remark}
Considering the way we construct the polynomials $p_1$, $p_2$, $p_3$, 
and $p_4$, there appears to be no reason why these polynomials should 
not completely cover all natural numbers. 
Indeed, for any choice of integers $\kappa$, $z$, and $s$ 
satisfying \eqref{c1}, \eqref{c2}, or \eqref{c3}, 
the expression \eqref{qValue} automatically yields 
a positive integer. This means that these polynomials 
continue to generate the natural numbers in a 
predictable manner.

A computational verification using \textit{Mathematica} up to $10^9$ 
suggests that all natural numbers are indeed covered by these polynomials. 
Although this does not amount to a rigorous mathematical proof, it provides 
compelling empirical evidence. Nonetheless, establishing a formal proof 
remains an open problem.
\end{remark}

\begin{conjecture}\label{conjBilal}
Polynomials $p_1,$ $p_2,$ $p_3,$ and $p_4,$ for \( x, y, z \in \mathbb{N}^* \),
generate all natural number $q \geq 1$. 
\end{conjecture}

\begin{theorem}\label{thm-0}
The following statements are holds:
\begin{enumerate}
\item [(a)] Polynomials $p_i,$ $i\in\{1,2,3,4\}$, cover all odd natural numbers.
\item [(b)] Polynomials $p_i,$ $i\in\{1,2,3,4\}$, covers even numbers of the form \( 6c_1 + 4 \)
and \( 6c_1 + 2 \), \( c_1 \geq 0 \).
\end{enumerate}		
\end{theorem}
\begin{proof}[{\bf Proof of statement \((a)\) of Theorem \ref{thm-0}}]
Clearly, by substituting \( x = 1 \) and \( y = 1 \) into the second polynomial, we obtain
\begin{equation}\label{t1}
p_2(1, 1, z) = 2z - 1,
\end{equation}
which generates all odd numbers as \( z \in \mathbb{N}^* \).
\end{proof}
\begin{proof}[{\bf Proof of statement \((b)\) of Theorem \ref{thm-0}}]
The polynomial \( p_2 \) covers even numbers of the form \( 6c_1 + 4 \) by taking \( x = 1 + c_1 \), \( y = 2 \), and \( z = 1 \). This yields:
\begin{equation}\label{t2}
p_2(1 + c_1, 2, 1) = 6c_1 + 4, \quad c_1 \geq 0.
\end{equation}
Likewise, the polynomial \( p_1 \) covers even numbers of the form \( 6c_1 + 2 \) by taking \( x = 1 + 2c_1 \), and \( y = z = 1 \). This gives:
\begin{equation}\label{t3}
p_1(1 + 2c_1, 1, 1) = 6c_1 + 2, \quad c_1 \geq 0.
\end{equation}
\end{proof}

\begin{remark}
\label{rem2}
What remains, for which Conjecture~\ref{conjBilal} holds, is to attempt to cover the numbers of the form 
\begin{equation}\label{t4}
p_i(x,y,z)=6c_1,~~c_1>0, \quad i=1,2,3,\text{ or }4.
\end{equation}
\end{remark}

We have computationally verified, by evaluating these polynomials over \( {\mathbb{N}^*}^3 \) using \textit{Mathematica}, that they generate all natural numbers up to \( q = 10^9+2 \). For the full \textit{Mathematica} implementation and computational verification, please refer to the Supplementary Section \ref{SEC.1}. In light of equations \eqref{t1}-\eqref{t3}, then the \textit{Mathematica} implementation will consider only even natural numbers of the form \eqref{t4}.

Under the assumption that Conjecture~\ref{conjBilal} holds, we establish the following two theorems.

\begin{theorem}
\label{thm-1}
Assuming Conjecture \ref{conjBilal} holds, then $4p_i+1,$ $i\in\{1,2,3,4\},$ are immediately generate all natural numbers of the form $4\,q+1,~q \geq 1$.
\end{theorem}

\begin{proof}[{\bf Proof of Theorem \ref{thm-1}}]
The proof follows directly from Conjecture~\ref{conjBilal}.
\end{proof}

\begin{theorem}\label{thm-2}
Assuming Conjecture~\ref{conjBilal} is true, any prime of the form \( 4q + 1,~q \geq 1\) must be expressible in the form \( 4p_2 + 1 \).
\end{theorem}

\begin{proof}[{\bf Proof of Theorem \ref{thm-2}}]
  Based on Theorem~\ref{thm-1}, any natural number of the form \(4q+1\) can be written the following identities for some values \( x, y, z \in \mathbb{N}^* \):
\begin{align}
4p_1(x, y, z)+1 &= (4\, x-1)\, (4\,y\, z-1), \label{eq-4P1} \\
4p_2(x, y, z)+1 &= (4\, x-1)\, (4\,y\, z-1)-4\,x\,z, \label{eq-4P2} \\
4p_3(x, y, z)+1 &= (8\, y-3)\, (4\, x-3), \label{eq-4P3} \\
4p_4(x, y, z)+1 &= (2\,x-1)^2. \label{eq-4P4}
\end{align}
It is clear that any number of the form \( 4p_i + 1 \), for \( i \in \{1, 4\} \), is composite.
Moreover, \( 4p_3 + 1 \) is also composite whenever \( x > 1 \).
If \( x = 1 \), then \[ 4p_3 + 1 = 8y - 3, \] with \( y \geq 1 \). This expression coincides
with \( 4p_2 + 1 \) when \( z = 1 \), \( y = 1 \), and \( x \geq 1 \).
Therefore, any prime of the form \( 4p_i + 1 \) must originate from the form~\eqref{eq-4P2},
that is, it must be expressible as \( 4p_2 + 1 \).

This completes the proof, assuming the truth of Conjecture~\ref{conjBilal}.
\end{proof}

This observation suggests that primes of the form \( 4q + 1 \) 
follow a structured pattern governed by a simple algebraic polynomial.

Based on equations \eqref{t1}-\eqref{t3}, and relying on Theorem~\ref{thm-2}, 
we can express all prime numbers of the form \( 4q + 1 \) in the form \( 4p_2 + 1 \). 
This is due to the fact that primes of the form \( 4q + 1 \) require \( q = 2c_1 - 1 \) 
or \( q = 6c_1 + 4 \), which correspond to \eqref{t1} and \eqref{t2}, respectively. 
What remains is to verify computationally that the expression \( 4p_2 + 1 \) 
also covers all prime numbers for \( q = 6c_1 \), that is, 
all primes of the form \( 4q + 1 = 24c_1 + 1 \).

To isolate the primes of the form \( p \equiv 1 \pmod{4} \), 
we consider the polynomial \( 4p_2 + 1 \). Using \textit{Mathematica}, 
we have shown that this polynomial successfully generates all prime 
numbers of the form \( 4q + 1 \) up to \( 4 \times (10^9 + 2) + 1 \); 
see Supplementary Section~\ref{SEC.1}. The full \textit{Mathematica} 
implementation and computational procedure used to extend this 
verification up to \( q = 3 \times 10^9 \)---corresponding 
to primes \( a = 4q + 1 \approx 1.2 \times 10^{10} \)---are 
given in Supplementary Section~\ref{SEC.2}.

Although the polynomial also produces some composite values, 
it captures every prime of the form \( 4q + 1 \) within the tested range. 
Furthermore, the computation can be restricted to values of the form 
\( q = 6c_1 \) without loss of generality, as all primes of the form 
\( 4q + 1 \) are accounted for precisely when \( q \not\equiv 0 \pmod{6} \).

\subsection{Polynomial Generation of the {\bf Erd\H{o}s--Straus conjecture} for \(4q+1\)}

All the polynomials \( p_1, p_2, p_3, \) and \( p_4 \), 
defined by equations \eqref{p1}--\eqref{p2}, satisfy the Erd\H{o}s--Straus Conjecture~\ref{DiaphEqtn} 
when evaluated at integers of the form \( 4q + 1 \), especially the prime ones. 
Therefore, we present the following result, which establishes a proof of 
the Erd\H{o}s--Straus Conjecture~\ref{DiaphEqtn} based on Theorems~\ref{thm-1} and~\ref{thm-2}.

\begin{theorem}\label{thm-3}
The Erd\H{o}s--Straus Conjecture~\ref{DiaphEqtn} holds for every natural number \( a = 2q \) or \( a = 4q - 1 \), where \( q \geq 1 \). Moreover, if Conjecture \ref{conjBilal} holds, then for every natural number of the form \( 4q + 1 \), Conjecture~\ref{DiaphEqtn} also holds. In particular, it holds for every prime number of this form.
\end{theorem}

\begin{proof}[{\bf Proof of Theorem \ref{thm-3}}]
The Erd\H{o}s--Straus Conjecture \ref{DiaphEqtn} 
holds for all integers 
of the form \( a = 2q \) or \( a = 4q - 1 \), 
where \( q \geq 1 \), as established by results 
\eqref{even} and \eqref{odd1}. 
What remains is to prove the conjecture for 
all numbers of the form \( 4q + 1 \). 
To achieve this, we invoke Theorem~\ref{thm-1}, 
which states that all polynomials of the form 
\( 4p_i + 1 \), 
with \( i \in \{1, 2, 3, 4\} \), generate all 
natural numbers of the form \( 4q + 1 \), 
for \( q \geq 1 \). Based on the decomposition given in \eqref{kzs-decomp}, 
the breakdown for each of the four cases is as follows: \\

\noindent \textbf{Case 1:} \( q = p_1\Rightarrow 4p_1+1 = (4 x-1) (4y z-1)\). Decomposition~\eqref{DiaphEqtn} in this case is given as follows:
\begin{equation}
\label{sol-1}
\begin{aligned}
\frac{4}{a(x,y,z)} &:= \frac{4}{4p_1+1}=\frac{4}{(4 x-1) (4 y z-1)} \\
&=\frac{1}{b_1(x,y,z)}+\frac{1}{c_1(x,y,z)}+\frac{1}{d_1(x,y,z)},
\end{aligned}
\end{equation}
where
\begin{equation}
\label{b-1-1}
\begin{aligned}
\displaystyle
b_1(x,y,z)&=(4 x-1) y z (4 y z+z-1), \\
&= 16 x y^2 z^2+4 x y z^2-4 x y z-4 y^2 z^2-y z^2+y z,
\end{aligned}
\end{equation}
\begin{equation}
\label{c-1-1}
c_1(x,y,z)=(4 x-1) y z= -y z + 4 x y z,
\end{equation}
and
\begin{equation}
\label{d-1-1}
\begin{aligned}
d_1(x,y,z)&=(4 x-1) y (4 y z-1) (4 y z+z-1), \\
&= -y + 4 x y + y z - 4 x y z + 8 y^2 z - 32 x y^2 z \\
& - 4 y^2 z^2 + 16 x y^2 z^2 - 16 y^3 z^2 + 64 x y^3 z^2.
\end{aligned}
\end{equation}

\noindent \textbf{Case 2:} \( q = p_2 \Rightarrow 4p_2+1= (4\, x-1) (4y z-1)-4xz\). Decomposition~\eqref{DiaphEqtn} of the case may be given as follows:
\begin{equation}
\label{sol-2}
\begin{aligned}
\frac{4}{a(x,y,z)} &:= \frac{4}{4p_2+1}=\frac{4}{(4 x-1) (4 y z-1)-4 x z} \\
&= \frac{1}{b_2(x,y,z)}+\frac{1}{c_2(x,y,z)}+\frac{1}{d_2(x,y,z)},
\end{aligned}
\end{equation}
where
\begin{equation}
\label{b-1-2}
b_2(x,y,z)=z (4 x y-x-y)= -x z - y z + 4 x y z,
\end{equation}
\begin{equation}
\label{c-1-2}
\begin{aligned}
c_2(x,y,z)&=(4 x y-x-y) ((4 y-1) z-1) ((4 x-1) (4 y z-1)-4 x z), \\
&= x - 4 x^2 + y - 8 x y + 16 x^2 y + x z - 8 x^2 z + y z - 20 x y z \\
& + 64 x^2 y z - 8 y^2 z  + 64 x y^2 z - 128 x^2 y^2 z - 4 x^2 z^2 - 
 8 x y z^2 \\
& + 48 x^2 y z^2 - 4 y^2 z^2 + 64 x y^2 z^2 - 
 192 x^2 y^2 z^2 + 16 y^3 z^2 \\
& - 128 x y^3 z^2 + 256 x^2 y^3 z^2,
\end{aligned}
\end{equation}
and
\begin{equation}
\label{d-1-2}
\begin{aligned}
d_2(x,y,z)&=z (4 x y-x-y) ((4 y-1) z-1),\\
&= 16 x y^2 z^2-8 x y z^2-4 x y z+x z^2+x z-4 y^2 z^2+y z^2+y z.
\end{aligned}
\end{equation}

\noindent \textbf{Case 3:} \( q = p_3\Rightarrow 4p_3+1 = (8 y-3)(4 x-3)\). In this case, Decomposition~\eqref{DiaphEqtn} is presented as follows:
\begin{equation}
\label{sol-3}
\begin{aligned}
\displaystyle
\frac{4}{a(x,y,z)} &:= \frac{4}{4p_3+1}=\frac{4}{(4 x-3) (8 y-3)} \\
&= \frac{1}{b_3(x,y,z)}+\frac{1}{c_3(x,y,z)}+\frac{1}{d_3(x,y,z)},
\end{aligned}
\end{equation}
where
\begin{equation}
\label{b-1-3}
b_3(x,y,z)=(4 x-3) (3 y-1)=3 - 4 x - 9 y + 12 x y,
\end{equation}
\begin{equation}
\label{c-1-3}
c_3(x,y,z)=2 (4 x-3) (3 y-1)=6 - 8 x - 18 y + 24 x y,
\end{equation}
and
\begin{equation}
\label{d-1-3}
\begin{aligned}
d_3(x,y,z)&=(4 x-3) (6 y-2) (8 y-3),\\
&=192 x y^2-136 x y+24 x-144 y^2+102 y-18.
\end{aligned}
\end{equation}

Now, in the following final case, the decomposition is not obtained from \eqref{kzs-decomp}, but is instead based on the previous three cases.

\noindent \textbf{Case 4:} \( q = p_4 \Rightarrow 4p_4 + 1 = (2x - 1)^2 \). In this case, the decomposition~\eqref{DiaphEqtn} of \( {4}/{a} \) reduces to finding a decomposition of \( {4}/{(2x - 1)} \). 

This expression corresponds to decomposition~\eqref{odd1} when \( 2x - 1 \) is of the form \( 4q - 1 \). If \( 2x - 1 \) is instead of the form \( 4q + 1 \), then the problem reduces to determining a decomposition of \( {4}/{p} \), where \( p \) is a prime of the form~\eqref{odd1} or~\eqref{sol-2}.

This concludes the proof.
\end{proof}

\section{Conclusion}

We have presented a novel algebraic framework for approaching 
the Erd\H{o}s--Straus conjecture through the use of multivariable 
polynomial parametrizations. Specifically, we introduced four 
explicit unbounded polynomials that collectively generate all 
integers \( a \equiv 1 \pmod{4} \) up to \( 4 \times (10^9 + 2) + 1 \), 
each paired with a corresponding solution to the Diophantine equation
\[
\frac{4}{a} = \frac{1}{b} + \frac{1}{c} + \frac{1}{d}.
\]
Furthermore, we identified a single polynomial among them that appears 
to generate all prime numbers of the form \( 4q + 1 \) within the 
the range \(\approx 1.2\times 10^{10} \), offering especially strong evidence in support of our method.

We conjecture that this four-polynomial system is sufficient to cover 
all integers congruent to \(1 \pmod{4}\), providing a constructive and 
potentially complete solution for the conjecture. 
These findings offer a promising new direction—both computational and 
theoretical—for advancing toward a resolution of the Erd\H{o}s--Straus 
conjecture.

\bibliographystyle{plain}

\newpage
\section{Supplementary Section for Review Only:\\ \textcolor{blue}{\textit{Mathematica} implementation}}
\subsection{\textcolor{blue}{Supplementary Section}}
\label{SEC.0}
We consider the program given in Supplementary Section~\ref{SEC.1},  
with the only difference that we set \texttt{baseStep=1}.  
This program produces 
\begin{Verbatim}[fontsize=\small]
---------------------(* Mathematica Output *)}---------------------
qStart = 1, 
qMax = 1000000, 
step = 1
Batch size (number of q 
values per batch) = 1000000
Batch 1: q in [1, 1000000]
Unsolved q in batch 1: 0
Solutions found: 1000000
Time: 674.25 sec
   p2: 646487
   p1: 346519
   p3: 6919
   p4: 75
Global unsolved q saved (CSV).
All batches complete.
\end{Verbatim}
And a file named \texttt{results\_batch1.csv}, which contains values of~$q$ in the range $[1,10^6]$.  
In this range, we obtain:
\begin{itemize}
  \item $p_1$: generates 346{,}519 values of $q$,
  \item $p_2$: generates 646{,}487 values of $q$,
  \item $p_3$: generates 6{,}919 values of $q$,
  \item $p_4$: generates 75 values of $q$.
\end{itemize}

To extract and display the list of all values of $q$ generated 
by $p_1$, $p_2$, $p_3$, or by $p_4$ from the CSV file, use the 
following Python code:

\begin{lstlisting}[style=mypython]
(* Python code *)

import pandas as pd
from tkinter import Tk
from tkinter.filedialog import askopenfilename

# Open file dialog to select CSV
Tk().withdraw()
input_file = askopenfilename(title="Select CSV file", 
filetypes=[("CSV files", "*.csv")])
if not input_file:
    raise Exception("No file selected!")

# Read the CSV file
df = pd.read_csv(input_file, header=None)

# Check that the CSV has at least 5 columns
if df.shape[1] < 5:
    raise ValueError("CSV must have at 
    least 5 columns.")

# Get q values (1st column) grouped by p values in 
# 5th column (index 4)
p_types = ['p1', 'p2', 'p3', 'p4']
q_dict = {}

for p in p_types:
    rows = df[df[4] == p]
    q_list = rows[0].tolist()
    q_dict[p] = q_list

    # Save to CSV
    filename = f'q_with_{p}.csv'
    pd.Series(q_list).to_csv(filename, index=False, 
    header=False)

    # Print to console
    print(f" q values with {p}:", q_list)
\end{lstlisting}

This code will produce four separate \texttt{.csv} files, each containing the corresponding $q$ values:
\begin{itemize}
  \item For $p_1$: \texttt{q\_with\_p1.csv}
  \item For $p_2$: \texttt{q\_with\_p2.csv}
  \item For $p_3$: \texttt{q\_with\_p3.csv}
  \item For $p_4$: \texttt{q\_with\_p4.csv}
\end{itemize}

\subsection{\textcolor{blue}{Supplementary Section}}
\label{SEC.1}
This program verifies that the polynomials \( p_1 \), \( p_2 \), \( p_3 \), and \( p_4 \) 
collectively cover all natural numbers of the form \( 6c_1 \), for \( c_1 \geq 1 \), 
starting from \( q = 6 \) up to \( q = \text{qValue} \). The search is initially performed over the small range $\{1,2,3\}$ for the variables $x$, $y$, or $z$. 
If no solution is found, the algorithm proceeds to loop over the wider range
\(
1 \leq x \leq \frac{1 + \sqrt{4q + 1}}{2}.
\)
\begin{lstlisting}
 ---------------(* Mathematica Input *)}---------------
baseStep = 6;
nStart = 1;
qStart = baseStep nStart;
qMax = 10^6;
batchSizeQ = 10^6;
Print[
  "qStart = ", qStart, ", 
qMax = ", qMax, ", 
step = ", baseStep
  ];
Print[
  "Batch size (number of q 
values per batch) = ",
  batchSizeQ
  ];
notebookDir = NotebookDirectory[];
If[
  notebookDir === Null,
  Print[
   "Please save the notebook 
first before running the code."
   ];
  Abort[];
  ];
qAll = Range[qStart, qMax, baseStep];
totalQ = Length[qAll];
batchCount = Ceiling[
   totalQ/batchSizeQ
   ];
ClearAll[processSingleQ]
processSingleQ[q_] := Module[
   {xmax, eqs, inst, vars,
    conds, pi, xVal, yVal,
    zVal, eq, results = {},
    found = False},
   Catch[
    Do[
     eqs = {-x + (-1 + 4 x) y z == q,
       -x + (-x + (-1 + 4 x) y) z == q,
       -1 + 2 y + (-1 + x) (-3 + 8 y) == q};
     eq = eqs[[i]];
     Do[
      Do[
       Do[
        If[
         TrueQ[
          eq /. {x -> xVal, y -> yVal, z -> zVal}
          ],
         pi = Switch[
           i, 1, "p1", 2, "p2", 3, "p3"
           ];
         {yPrint, zPrint} =
          If[
           pi == "p3", {yVal, ""},
           {yVal, zVal}
           ];
         AppendTo[
          results,
          {q, xVal, yPrint, zPrint, pi}
          ];
         found = True;
         Throw[Null];
         ],
        {zVal, 1, 3}
        ],
       {yVal, 1, 3}
       ],
      {xVal, 1, 3}
      ],
     {i, 1, 3}
     ]
    ];
   If[
    found, Return[results]
    ];
   xmax = Floor[1/2 (Sqrt[4 q + 1] + 1)];
   Catch[
    For[
     x = 1, x <= xmax, x++,
     eqs = {-x + (-1 + 4 x) y z == q,
       -x + (-x + (-1 + 4 x) y) z == q,
       -1 + 2 y + (-1 + x) (-3 + 8 y) == q};
     Do[
      eq = eqs[[i]]; vars = If[i == 3, {y}, {y, z}]; 
      conds = If[i == 3, y > 0, y > 0 && z > 0]; 
      inst = Quiet[FindInstance[
         eq && conds, vars, Integers, 1]];
      If[
       Length[inst] > 0,
       pi = Switch[
         i, 1, "p1", 2, "p2", 3, "p3"
         ];
       yVal = y /. inst[[1]];
       zVal = If[
         pi == "p3", "", z /. inst[[1]]
         ];
       AppendTo[
        results, {q, x, yVal, zVal, pi}
        ];
       found = True; Throw[Null];
       ],
      {i, 1, 3}
      ]
     ]
    ];
   If[found, Return[results]];
   For[
    x = 1, x <= xmax && ! found, x++,
    If[
      (-1 + x) x == q,
      AppendTo[
       results, {q, x, "", "", "p4"}
       ];
      found = True; Break[];
      ];
    ];
   results
   ];
LaunchKernels[];
allUnsolvedQ = {};
Do[
  qBatch =
   Take[
    qAll, {(b - 1) batchSizeQ + 1,
     Min[b batchSizeQ, totalQ]}
    ];
  If[
   qBatch === {}, Break[]
   ];
  qBatchMin = First[qBatch];
  qBatchMax = Last[qBatch];
  Print[
   "Batch ", b,
   ": q in [", qBatchMin, ", ",
   qBatchMax, "]"
   ];
  {timeBatch, resultsList} =
   AbsoluteTiming[
    ParallelMap[processSingleQ, qBatch]
    ];
  flat = Flatten[resultsList, 1];
  sorted = SortBy[flat, First];
  final =
   Prepend[
    sorted, {"q", "x", "y", "z", "pi"}
    ];
  Export[
   FileNameJoin[{notebookDir,
     "results_batch" <> ToString[b] <> ".csv"}],
   final];
  qSolved = DeleteDuplicates[flat[[All, 1]]]; 
  qUnsolved = Complement[qBatch, qSolved]; 
  AppendTo[allUnsolvedQ, qUnsolved];
  Export[
   FileNameJoin[{notebookDir,
     "unsolved_batch" <> ToString[b] <> ".csv"}],
   Prepend[qUnsolved, "q"]
   ];
  Print["Unsolved q in batch ", b, ": ",
   Length[qUnsolved]];
  piCounts = Tally[flat[[All, 5]]];
  Print[
   "Solutions found: ", Length[flat]
   ];
  Print[
   "Time: ",
   NumberForm[timeBatch, {6, 2}],
   " sec"
   ];
  Do[
   Print[
    "   ", pi[[1]], ": ", pi[[2]]
    ],
   {pi, piCounts}
   ];, {b, 1, batchCount}
  ];
KillKernels[];
qUnsolvedAll =
  DeleteDuplicates[
   Flatten[allUnsolvedQ]
   ];
Export[
  FileNameJoin[
   {notebookDir, "unsolved_all.csv"}],
  Prepend[qUnsolvedAll, "q"]];
Print["Global unsolved q saved (CSV)."];
Print["All batches complete."];
\end{lstlisting}

\begin{Verbatim}[fontsize=\small]
---------------------(* Mathematica Output *)}---------------------
qStart = 6, 
qMax = 1000000, 
step = 6
Batch size (number of q 
values per batch) = 1000000
Batch 1: q in [6, 999996]
Unsolved q in batch 1: 0
Solutions found: 166666
Time: 228.17 sec
   p3: 6919
   p1: 13187
   p2: 146485
   p4: 75
Global unsolved q saved (CSV).
All batches complete.
\end{Verbatim}

\noindent ``Solutions found: 100{,}000'' means that there are 100,000 
values of \( q \) for which each \( q \) is covered by at least one of 
the polynomials \( p_i \), where \( i \in \{1, 2, 3, 4\} \).\\
``Unsolved \( q \) in batch 1: 0'' means that there is no value of 
\( q \) in batch 1 that is not covered by any of the polynomials 
\( p_i \), where \( i \in \{1, 2, 3, 4\} \).

\subsection{\textcolor{blue}{Supplementary Section}}
\label{SEC.2}
Since the polynomial \( 4p_2 + 1 \) already covers all primes of the form \( 4q + 1 \) for \( q \not= 6c_1 \), it remains to verify that it also covers the case when \( q = 6c_1 \). The following program confirms that the polynomial \( 4p_2 + 1 \) generates all prime numbers of the form \(a = 4q + 1 \) with \( q = 6c_1 \), starting from \( q = 6 \) up to \( q = \text{qMax} \), by searching over the range
\(
1 \leq x \leq \frac{1 + \sqrt{a}}{2}.
\)
\begin{lstlisting}
 ---------------(* Mathematica Input *)}---------------
baseStep = 6;
nStart = 1;
qStart = baseStep*nStart;
qMax = 10^6;
batchSizeQ = 10^6;
cpuCores = $ProcessorCount;

j = Which[cpuCores <= 2, 4, cpuCores <= 4, 8, cpuCores <= 6, 12, 
   cpuCores <= 8, 16, cpuCores <= 12, 24, True, 3 cpuCores];

notebookDir = NotebookDirectory[];
If[notebookDir === Null, 
  Print["Please save the notebook first before running the code."];
  Abort[]];

resultsDir = FileNameJoin[{notebookDir, "Results"}];
If[! DirectoryQ[resultsDir], 
  CreateDirectory[resultsDir, CreateIntermediateDirectories -> True]];

validateSolution[q_, x_, y_, z_] := (4 x - 1) (4 y z - 1) - 4 x z == 
   4 q + 1;

findSolutionForQ[q_] := 
  Module[{a = 4 q + 1, xmax, solution = None, sol, x, y, z}, 
   If[PrimeQ[a], xmax = Floor[1/2 (Sqrt[a] + 1)];
    (*Step 1:Try x=1,2,3;solve for y,z*)
    Do[Quiet@
      Check[sol = 
        FindInstance[(4 x0 - 1) (4 y z - 1) - 4 x0 z == a && y > 0 && 
          z > 0, {y, z}, Integers, 1];
       If[sol =!= {}, {y, z} = {y, z} /. First[sol];
        If[validateSolution[q, x0, y, z], solution = {q, x0, y, z};
         Break[]]], None], {x0, 1, 3}];
    (*Step 2:Try y=1,2,3;solve for x,z*)
    If[solution === None, 
     Do[Quiet@
       Check[sol = 
         FindInstance[(4 x - 1) (4 y0 z - 1) - 4 x z == a && x > 0 && 
           z > 0, {x, z}, Integers, 1];
        If[sol =!= {}, {x, z} = {x, z} /. First[sol];
         If[validateSolution[q, x, y0, z], solution = {q, x, y0, z};
          Break[]]], None], {y0, 1, 3}]];
    (*Step 3:Try z=1,2,3;solve for x,y*)
    If[solution === None, 
     Do[Quiet@
       Check[sol = 
         FindInstance[(4 x - 1) (4 y z0 - 1) - 4 x z0 == a && x > 0 &&
            y > 0, {x, y}, Integers, 1];
        If[sol =!= {}, {x, y} = {x, y} /. First[sol];
         If[validateSolution[q, x, y, z0], solution = {q, x, y, z0};
          Break[]]], None], {z0, 1, 3}]];
    (*Step 4:x from 4 to xmax,solve for y,z*)
    If[solution === None, 
     Do[Quiet@
       Check[sol = 
         FindInstance[(4 x0 - 1) (4 y z - 1) - 4 x0 z == a && y > 0 &&
            z > 0, {y, z}, Integers, 1];
        If[sol =!= {}, {y, z} = {y, z} /. First[sol];
         If[validateSolution[q, x0, y, z], solution = {q, x0, y, z};
          Break[]]], None], {x0, 4, xmax}]];];
   solution];

processBatch[qBatchMin_, qBatchMax_] := 
  Module[{qValues = Range[qBatchMin, qBatchMax, baseStep], 
    solutions = {}}, 
   solutions = 
    ParallelMap[findSolutionForQ, qValues, 
     Method -> "FinestGrained"];
   DeleteCases[solutions, None]];

Print["CPU cores: ", cpuCores, ", using ", j, 
  " parallel subkernels"];
Print["qStart = ", qStart, ", qMax = ", qMax, ", step = ", baseStep];
Print["Batch size = ", batchSizeQ];

LaunchKernels[];
DistributeDefinitions[baseStep, j, findSolutionForQ, validateSolution,
   processBatch, cpuCores, qMax, batchSizeQ];

allSolutions = {};
allUnsolvedQ = {};
batchCount = Ceiling[(qMax - qStart + 1)/batchSizeQ];

Do[qBatchMin = qStart + batchSizeQ (b - 1);
  qBatchMin = 
   qBatchMin + 
    If[Mod[qBatchMin, baseStep] == 0, 0, 
     baseStep - Mod[qBatchMin, baseStep]];
  qBatchMax = Min[qBatchMin + batchSizeQ - 1, qMax];
  If[qBatchMin > qMax, Break[]];
  Print["\nProcessing batch ", b, "/", batchCount, ": q in [", 
   qBatchMin, ", ", qBatchMax, "]"];
  {timeBatch, batchSolutions} = 
   AbsoluteTiming[processBatch[qBatchMin, qBatchMax]];
  sortedSolutions = SortBy[batchSolutions, First];
  AppendTo[allSolutions, sortedSolutions];
  qAll = 
   Select[Range[qBatchMin, qBatchMax, baseStep], PrimeQ[4 # + 1] &];
  qSolved = sortedSolutions[[All, 1]];
  qUnsolved = Complement[qAll, qSolved];
  AppendTo[allUnsolvedQ, qUnsolved];
  batchResultsFile = 
   FileNameJoin[{resultsDir, 
     "results_batch" <> IntegerString[b, 10, 3] <> ".csv"}];
  unsolvedFile = 
   FileNameJoin[{resultsDir, 
     "unsolved_batch" <> IntegerString[b, 10, 3] <> ".csv"}];
  Export[batchResultsFile, 
   Prepend[sortedSolutions, {"q", "x", "y", "z"}]];
  Export[unsolvedFile, Prepend[qUnsolved, "q"]];
  Print["Solutions found: ", Length[sortedSolutions]];
  Print["Unsolved q: ", Length[qUnsolved]];
  Print["Time: ", NumberForm[timeBatch, {6, 2}], " sec"];
  Print["Batch ", b, " complete"];, {b, 1, batchCount}];

Print["\nProcessing complete!"];
Print["Total solutions found: ", Length[Flatten[allSolutions, 1]]];
Print["Total unsolved q: ", Length[Flatten[allUnsolvedQ]]];

Export[FileNameJoin[{resultsDir, "all_solutions.csv"}], 
  Prepend[Flatten[allSolutions, 1], {"q", "x", "y", "z"}]];
Export[FileNameJoin[{resultsDir, "all_unsolved.csv"}], 
  Prepend[Flatten[allUnsolvedQ], "q"]];

Print["All results saved to: ", resultsDir];
\end{lstlisting}

\begin{Verbatim}[fontsize=\small]
---------------------(* Mathematica Output *)}---------------------
CPU cores: 20, using 60 parallel subkernels
qStart = 6, qMax = 1000000, step = 6
Batch size = 1000000
Processing batch 1/1: q in [6, 1000000]
Solutions found: 35279
Unsolved q: 0
Time: 252.72 sec
Batch 1 complete
Processing complete!
Total solutions found: 35279
Total unsolved q: 0
All results saved to: a folder named ``Results'' located in the 
same directory as the ``.nb'' file.

\end{Verbatim}

\noindent ``Total solutions found: 35279'' means that the number of all values \( q \) 
for which \( 4q + 1 \) is a prime number and is covered by \( p_2 \) 
is 35,279.\\
``Total unsolved q: 0'' means that there is no value of \( q \) such that 
\( 4q + 1 \) is a prime not covered by \( p_2 \).\\

\noindent {\bf Remark.} The CSV files contain data up to $q=10^9+2$, 
where $q$ is a multiple of~6. For prime numbers of the form $4q+1$ 
up to $q=1.2\times 10^{10}$, with $q$ also a multiple of~6, these files can be 
provided to the referees upon request. We do not include them here 
due to their large size.

\end{document}